\documentclass[11pt,a4paper]{article}

\usepackage{amsthm,amsmath,amsfonts,amssymb}
\usepackage[english]{babel}
\usepackage{graphicx}
\usepackage{verbatim}
\usepackage{bm}
\usepackage{color}
\usepackage{soul}

\newcommand\blfootnote[1]{%
  \begingroup
  \renewcommand\thefootnote{}\footnote{#1}%
  \addtocounter{footnote}{-1}%
  \endgroup
}

\newtheorem{proposition}{Proposition}
\newtheorem{theorem}{Theorem}
\newtheorem{lemma}{Lemma}
\newtheorem{corollary}{Corollary}

\newtheorem{remark}{Remark}
\newtheorem{claim}{Claim}

\title{Sufficient conditions
for a digraph to admit a $(1,\leq\ell)$-identifying code
}

\author{ C. Balbuena$^{1}$, C. Dalf\'o$^{2}$, B. Mart\'\i nez-Barona$^{1}$
 \\[2ex]
$^1${\footnotesize Departament d'Enginyeria Civil i Ambiental, Universitat
Polit\`ecnica de Catalunya, }\\
{\footnotesize %Campus Nord, Edifici C2, C/ Jordi Girona 1 i 3 E-08034
Barcelona, Catalonia, Spain.}\\[2ex]
$^2${\footnotesize Departament de Matem\`atica, Universitat
 de Lleida, }\\
{\footnotesize %Campus Nord, Edifici C3, C/ Jordi Girona 1 i 3 E-08034
Igualada (Barcelona), Catalonia.}\\[2ex]
{\footnotesize e-mails: \{m.camino.balbuena, berenice.martinez\}@upc.edu, cristina.dalfo@matematica.udl.cat}
}

\date{}

\begin{document}

\maketitle

\blfootnote{
\begin{minipage}[l]{0.3\textwidth} \includegraphics[trim=10cm 6cm 10cm 5cm,clip,scale=0.15]{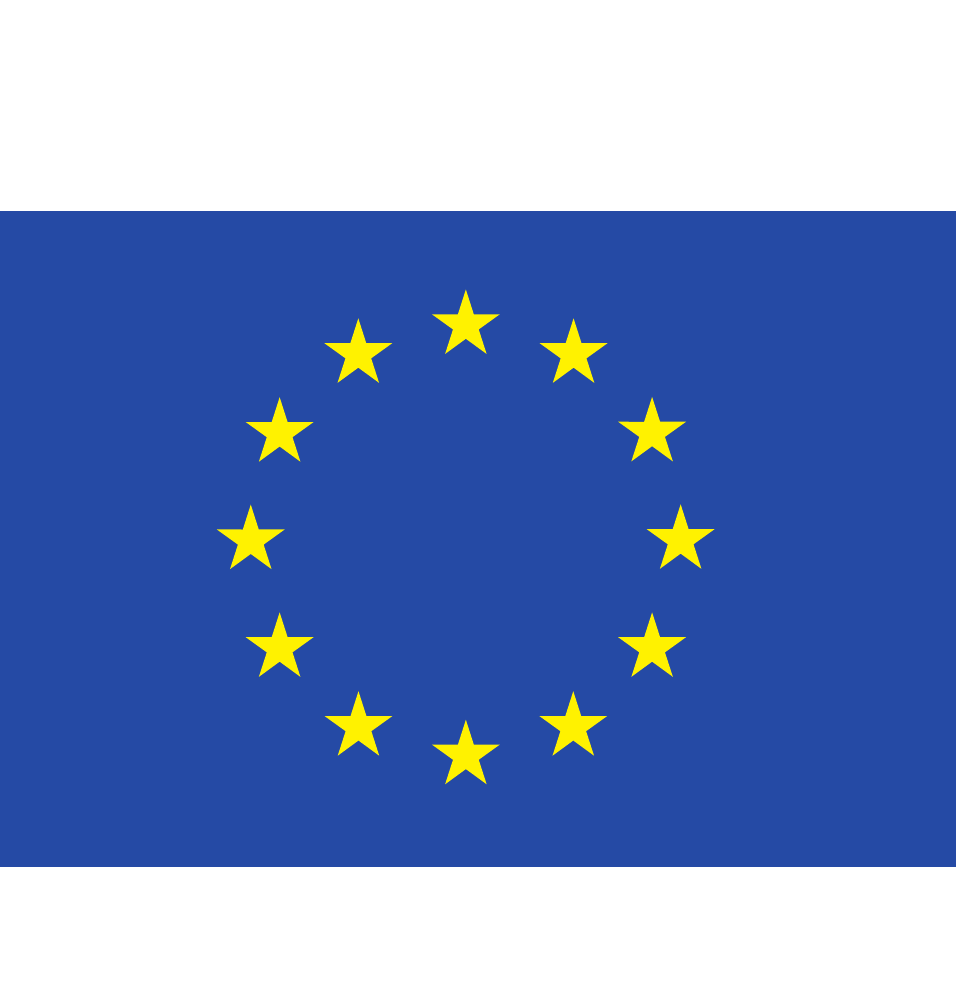} \end{minipage}  \hspace{-2cm} \begin{minipage}[l][1cm]{0.79\textwidth}
   The last two authors have also received funding from the European Union's Horizon 2020 research and innovation programme under the Marie Sk\l{}odowska-Curie grant agreement No 734922.
  \end{minipage}}

\begin{abstract}
A $(1,\le \ell)$-identifying code in a digraph $D$ is a  subset $C$ of
vertices of $D$ such that all distinct subsets of vertices  of cardinality at most $\ell$ have distinct closed in-neighbourhoods
within $C$. In this paper, we give some sufficient conditions for a digraph of minimum in-degree  $\delta^-\ge 1$ to admit a $(1,\le \ell)$-identifying code  for $\ell=\delta^-, \delta^-+1$. As a corollary,
we
obtain the result by Laihonen that states that
a  graph of minimum degree $\delta\ge 2$ and girth at least 7 admits a $(1,\le \delta)$-identifying code. Moreover,
we prove that every $1$-in-regular digraph has a $(1,\le 2)$-identifying code if and only if the girth of the digraph is at least 5. We also characterize all the 2-in-regular digraphs admitting a $(1,\le \ell)$-identifying code for $\ell=2,3$.
\end{abstract}



%%%%%%%%%%%%%%%%%%%%%%%%%%%%%%%%%%%%%%%%%%%%%%%%%%%%%%%%%%%%%%%%%%%%%%%%%%%%%%%%%%%%%%%%%%%%%%%%%%%%%%%%%%%%%%%%%%%%%%%%%%%%%%%%%%%
%%%%%%%%%%%%%%%%%%%%%%%%%%%%%%%%%%%%%%%%%%%%%%%%%%%%%%%%%%%%%%%%%%%%%%%%%%%%%%%%%%%%%%%%%%%%%%%%%%%%%%%%%%%%%%%%%%%%%%%%%%%%%%%%
\section{Introduction}

The aim of this paper is to study identifying codes in digraphs.
We  consider  simple digraphs (or directed graphs) without loops or multiple edges. Unless otherwise
stated, we follow the book by Bang-Jensen and Gutin \cite{BG07} for terminology and definitions.

Let $D=(V,A)$ be a digraph with vertex set $V(D)=V$ and arc set $A(D)=A$.
  A vertex $u$ is \emph{adjacent to} a vertex $v$ if  $(u,v)\in A$. If both arcs $(u,v),(v,u)\in A$,  then we say that   they form a \emph{digon}. A digraph is \emph{symmetric} if $(u,v)\in A$ implies $(v,u)\in A$. A digon is often said a \emph{symmetric arc} of $D$. A digraph $D$ is said to be \emph{oriented graph} if $D$ has no digon. The \emph{girth} $g$ of a digraph is the length of a shortest directed cycle. Hence, an  oriented graph has girth $g\ge 3$. Moreover, observe that every graph $G$ with vertex set $V$ and edge set $E$ can be seen as a symmetric digraph denoted by $ \overset{\text{$\leftrightarrow $}}{G}$, replacing each edge $uv\in E$ by the digon $(u,v)$ and $(v,u)$.
  The \emph{out-neighborhood} of a vertex $u$ is $N^+(u)=\{v\in V: (u,v)\in A\}$ and the \emph{in-neighborhood} of  $u$ is $N^-(u)=\{v\in V: (v,u)\in A\}$. The {\em closed in-neighbourhood} of $u $ is $N^-[u]=\{u\}\cup N^-(u)$.
  The
 \emph{out-degree} of $u$ is $d^{+}(u)=|N^+(u)|$ and its \emph{in-degree}  $d^{-}(u)=|N^-(u)|$. We denote by $\delta^+(D)$ the minimum out-degree of the vertices in $D$, and by $\delta^-(D)$ the minimum in-degree. The minimum degree is $\delta(D)=\min\{\delta^+(D), \delta^-(D)\}$.
 A digraph $D$ is said to be $d$-\emph{in-regular} if $d^-(v)=d$ for all $v\in V$,  and  $d$\emph{-regular} if $d+(v)=d^-(v)=d$ for all $v\in V$.

Given a vertex subset $U\subseteq V$, let $N^-[U]=\bigcup_{u\in U}N^-[u]$.
 For a given integer $\ell\ge 1$, a
vertex subset $C\subseteq V$ is a {\em $(1,\leq\ell)$-identifying code} in $D$ when for all distinct subsets $X,Y\subseteq V$, with $1\le |X|,|Y|\leq\ell$, we have
\begin{equation}\label{code}
N^-[X]\cap C \neq N^-[Y]\cap C.
\end{equation}
The definition of a $(1,\leq\ell)$-identifying code for graphs was introduced by Karpovsky, Chakrabarty and Levitin~\cite{kcl98},
and it is obtained by omitting the superscript sign minus in the neighborhoods in (\ref{code}). Thus, the definition for digraphs  is a natural extension  of the concept of  $(1,\le \ell)$-identifying codes in
graphs. A $(1,\le 1)$-identifying code is known as an \emph{identifying code}.
Thus,
an identifying code of a graph is a  set of vertices such that
any two vertices of the graph have distinct closed neighborhoods within this set.
Identifying codes model fault-diagnosis in multiprocessor systems, and these are used in other
applications such as the design of emergency sensor networks. Identifying codes in graphs have received much more attention by researchers.
 Honkala and  Laihonen \cite{hl08} studied
 identifying codes in the king grid that are robust against edge deletions.
More recently,
identifying codes have been considered for vertex-transitive graphs and strongly regular graphs by Gravier et al. \cite{GPRSV15}, and for graphs of girth at least five by Balbuena, Foucaud and Hansberg \cite{bfh15}. Other results on identifying
codes in specific families of graphs, as well as on the smallest cardinality of an identifying code $C$,
can be seen in  Bertrand et al. \cite{bchl04}, Charon et al. \cite{cchl10},   Exoo et al. \cite{ejlr09,ejlr10}, and the online bibliography of   Lobstein \cite{Lon}.

A graph $G$ is said to admit a $(1,\le \ell)$-identifying code if there is a subset of vertices $C \subseteq V(G)$ such that $C$ is a  $(1,\le \ell)$-identifying code in $G$. Not all graphs admit $(1,\le \ell)$-identifying codes. For instance, Laihonen \cite{l08} pointed out
that a graph containing an isolated edge cannot admit a $(1,\le 1)$-identifying code, because
clearly,  if $uv \in  E(G)$ is isolated, then $N[u] = \{u, v\} = N[v]$.
In fact, a graph containing an isolated complete bipartite graph $K_{r,d}$, with $r\leq d$, cannot admit a $(1,\le d)$-identifying code.
It is not difficult to see that if $G$ admits a $(1,\le \ell)$-identifying
code, then $C = V$ is also a $(1,\le \ell)$-identifying code. Hence, a graph admits a $(1,\le \ell)$-identifying code
if and only if the sets $N[X]$ are mutually different for all $X \subseteq V $, with $|X| \le  \ell$.
Laihonen and Ranto \cite{lr01} proved that if $G$ is a connected graph with at least three vertices admitting
a $(1,\le \ell)$-identifying code, then the minimum degree is $\delta(G) \ge  \ell$. Gravier and Moncel \cite{gm05} showed the
existence of a graph with minimum degree exactly $\ell$ admitting a $(1,\le \ell)$-identifying code.
Laihonen \cite{l08} proved the following result.

\begin{theorem} \label{laiho} \cite{l08}
 Let $k\geq2$ be an integer.
 \begin{enumerate}
\item If a $k$-regular graph has girth $g\geq7$, then it admits a $(1,\leq k)$-identifying code.
\item If a $k$-regular graph has girth $g\geq5$, then it admits a $(1,\leq k-1)$-identifying code.
\end{enumerate}
\end{theorem}

Araujo et al. \cite{abmv11} characterized the bipartite $k$-regular graphs of girth at least 6 having a $(1,\leq k)$-identifying code.

Identifying codes for
digraphs  were considered by Charon et al. \cite{CGHL02,CGHLMM06}, and Frieze, Martin, Moncet et al. \cite{FMMRS07} studied identifying codes in random networks. Recently, Foucaud, Naserasr and Parreau~\cite{fnp13} studied identifying codes in digraphs under the name of separating sets, and they called  identifying codes to the separating sets that also are dominating sets.
These authors
characterized the finite digraphs that only admit their whole vertex set as an  identifying code in this meaning.

In this paper,  we   give some sufficient conditions for a digraph of minimum in-degree  $\delta^-\ge 1$ to admit a $(1,\le \ell)$-identifying code  for $\ell=\delta^-, \delta^-+1$. As a corollary,
we
obtain Theorem \ref{laiho}. Moreover,
we prove that every $1$-in-regular digraph has a $(1,\le 2)$-identifying code if and only if the girth of the digraph is at least 5. We also characterize all the 2-in-regular digraphs admitting a $(1,\le \ell)$-identifying code for $\ell=2,3$.

%%%%%%%%%%%%%%%%%%%%%%%%%%%%%%%%%%%%%%%%%%%%%%%%%%%%%%%%%%%%%%%%%%%%%%%%%%%%%%%%%%%%%%%%%%%%%%%%%%%%%%%%%%%%%%%%%%%%%%%%%%%%%%%%
%%%%%%%%%%%%%%%%%%%%%%%%%%%%%%%%%%%%%%%%%%%%%%%%%%%%%%%%%%%%%%%%%%%%%%%%%%%%%%%%%%%%%%%%%%%%%%%%%%%%%%%%%%%%%%%%%%%%%%%%%%%%%%%%
\section{Identifying codes}

In this paper we study the concept  of a $(1,\leq\ell)$-identifying code for digraphs given in (\ref{code}).
We begin by noting that if $C$ is a $(1,\leq \ell)$-identifying code in a digraph $D$, then the whole set of vertices $V$ also is. Thus, we have the following straightforward observation.

\begin{lemma}
\label{basic-lema}
A digraph $D=(V,A)$ admits some $(1,\leq\ell)$-identifying code if and only if
for all distinct subsets $X,Y\subseteq V$ with $|X|,|Y|\leq \ell$, we have
\begin{equation}
N^-[X] \neq N^-[Y].
\end{equation}
\end{lemma}

As already mentioned in the introduction, Laihonen and Ranto \cite{lr01} proved that if $G$ is a connected graph with at least three vertices admitting
a $(1,\le \ell)$-identifying code, then the minimum degree is $\delta(G) \ge  \ell$. We present the following similar result for digraphs.

\begin{proposition}
 \label{BoundDelta+1}
 Let $D$ be a digraph admitting a $(1,\leq \ell)$-identifying code. Let $u$ be a vertex such that $d^+(u)\geq 1$. Then, $\ell\leq d^-(u)+1$. Furthermore, if $u$ belongs to a digon, then $\ell\leq d^-(u)$.
\end{proposition}
\begin{proof}
Let $u\in V(D)$ be such that $d^+(u)\geq1$ and $v\in N^+(u)$. Then,
both sets $X=N^-(u)\cup\{u,v\}$ and $Y=N^-(u)\cup\{v\}$ have the same closed in-neighbourhood. Consequently, $\ell\leq d^-(u)+1$. Furthermore,
if $v\in N^-(u)$, then $X'=N^-(u)\cup\{ u\}$ and $Y'=N^-(u)$ have the same closed in-neighbourhood implying that $\ell\leq d^-(u) $.
\end{proof}

%\begin{proposition}
%  Let $D$ be a digraph admitting a $(1,\leq \ell)$-identifying code.
%If some vertex of $D$ of minimum in-degree $\delta^- $ has out-degree  at least one, then    $\ell\leq\delta^-+1$ by Proposition  \ref{BoundDelta+1}.
%\end{proposition}

%\begin{corollary}
% \label{Regular&OrientedBound}
% Let $D$ be a digraph admitting a $(1,\leq \ell)$-identifying code. If there exists a vertex $u$ of minimum in-degree $\delta^-\geq 2$ such that $d+(u)\geq1$,  then $\ell \leq \delta^-+1$.
%\end{corollary}
\begin{corollary}
\label{nodig}
Let  $D$ be a  digraph  of minimum in-degree $\delta^- $ admitting a $(1,\leq \delta^-+1)$-identifying code.   Then, any vertex $u$ with $d^-(u)=\delta^-$ does not lay on a digon.
%\begin{proof}
%Suppose that   $(u,x),(x,u)\in A(D)$ for some vertex $x$. The  set $X=N^-[u]$ of cardinality $\delta^-+1$   and the set $Y=N^-(u)$ of cardinality $\delta^-$ have the same closed in-neighbourhood. Therefore, $D$ does no admits a $(1,\leq \delta^-+1)$-identifying code, which is a contradiction.
%\end{proof}
\end{corollary}

\begin{corollary}
 \label{DeltaBound}
 Let $D$ be a digraph admitting a $(1,\leq \ell)$-identifying code. Then, $\ell \leq \min\{d^-(u)+1\mid u\in V(D)\mbox{ and } d^+(u)\geq1\}$.
\end{corollary}

We recall that a \emph{transitive tournament} of three vertices is denoted by $TT_3$, see $F_1$ of Figure~\ref{fig:conf-prohibides3}.
\begin{remark}
\label{TT3}
Let $D$ be a $TT_3$-free digraph. Then, for every arc $(x,y)$ of $D$, we have $N^-(x)\cap N^-(y)=\emptyset$ and $N^+(x)\cap N^+(y)=\emptyset$.
\end{remark}

\begin{remark}\label{twin}
Two distinct vertices  $u$ and $v$   of $D$ are called \emph{twins} if $N^-[u]=N^-[v]$.  Hence,  a digraph $D$ admits a $(1,\le 1)$-identifying code if and only if $D$ is twin-free.
\end{remark}

\begin{theorem}
\label{theo:ident-code-digraphs-altern}
Let $D$ be a twin-free digraph with minimum in-degree $\delta^-\ge 1$.
\begin{itemize}
\item[$(i)$] Suppose that $\delta^-\ge 2$ and $D$ does not  contain any subdigraph as $F_1$ nor $F_2$ of Figure~\ref{fig:conf-prohibides3}, then $D$ admits a $(1,\leq\ \delta^--1)$-identifying code.

\item[$(ii)$] Suppose that $D$ is an oriented graph and does not contain any subdigraph as $F_1$ nor $F_2$ of Figure~\ref{fig:conf-prohibides3}, then $D$ admits a $(1,\leq\ \delta^-)$-identifying code.

\item[$(iii)$]
If $D$ does not contain any subdigraph from $F_1$ to $F_9$ of Figure~\ref{fig:conf-prohibides3}, then $D$ admits a $(1,\leq \delta^-)$-identifying code.

\item[$(iv)$] Suppose that $\delta^-\ge 2$ and the vertices of in-degree $\delta^-$ do not lay on a digon. If $D$ does not contain any subdigraph as those of Figure~\ref{fig:conf-prohibides3}, then $D$ admits a $(1,\leq\delta^-+1)$-identifying code.

\item[$(v)$] Suppose that $\delta^-=1$ and the vertices of in-degree 1 do not lay on directed cycles of length less than five. If $D$ does not contain any subdigraph as $F_1$,  $F_3$, $F_4$, $F_5$, $F_6$ nor $F_{11}$ of  Figure~\ref{fig:conf-prohibides3}, then $D$ admits a $(1,\leq  2)$-identifying code.
\end{itemize}
\end{theorem}

\begin{figure}[t]
\vskip-.75cm
    \begin{center}
  \includegraphics[width=12.5cm]{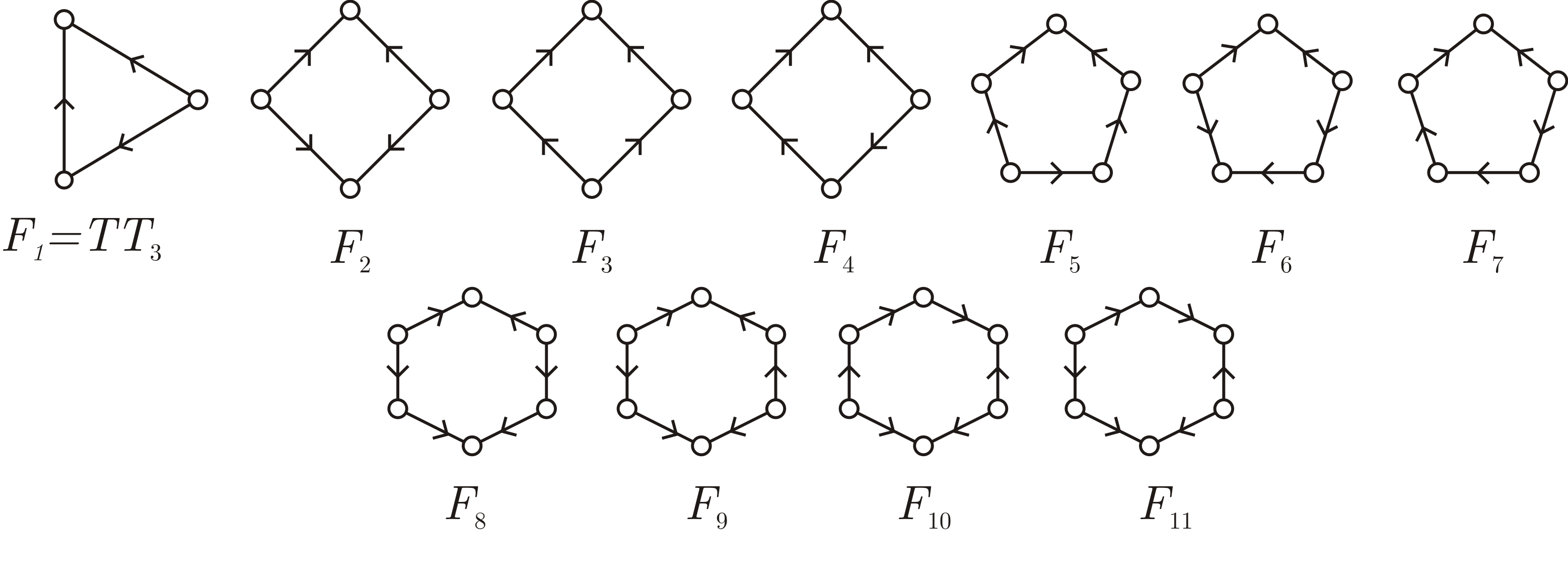}
  \end{center}
  %\vskip-5.75cm
  \caption{All the forbidden subdigraphs of Theorem~\ref{theo:ident-code-digraphs-altern}.}
  \label{fig:conf-prohibides3}
\end{figure}

%\textcolor[rgb]{0.00,0.50,1.00}{FALTARIA UN DIBUJO DE LA DEMOSTRACIÓN, MEJOR DAR FACILIDADES A LOS REFEREES}

\begin{proof} By Remark \ref{twin},  $D$ admits a $(1,\leq\  1)$-identifying code because $D$ is twin-free. In what follows, for brevity, we  made reference to the different cases $F_1$-$F_{11}$
of Figure \ref{fig:conf-prohibides3} without mentioning the figure.

 We reason  assuming that   $D$ does not admit a $(1,\leq\ \ell )$-identifying code with $\ell \in  \{ \delta^--1  , \delta^-, \delta^-+1\}$.
Then there are two different subsets  $X$ and $Y$ with $|Y|, |X|\le \ell$ such that $N^-[X] = N^-[Y]$. Let $x\in X\setminus Y$ and  $N^-(x)=\{v_1,\ldots,v_{\tau}\}$ for $\tau\ge \delta^-$.
As $N^-(x)\subseteq N^-[X]=N^-[Y]$, for all $v_i$,  $i=1,\ldots,\tau$, there exists a vertex $y_i\in Y$ such that $y_i\in N^+(v_i)$ or $y_i=v_i\in Y$. Moreover, all vertices $y_i$ are mutually different, since otherwise some subdigraph $F_1$ or $F_2$ would be contained in $D$.
Hence, $|Y|\ge \delta^-$, which contradicts the hypothesis of $(i)$, and the proof of $(i)$ is completed.

Observe that both $(ii)$ and $(iii)$ are proved if  $\delta^-=1$, so we may assume that $\delta^-\ge 2$ in these two cases.

We continue the proof assuming that $\ell \in  \{   \delta^-, \delta^-+1\}$.
Since $x\in  N^-[X]=N^-[Y]$,  there is $y\in Y$ such that $y\in N^+(x)$. Observe that
 $y\not\in N^-(x)$ because by hypothesis of $(ii)$ the digraph is an oriented graph. Moreover, $y$ is different from each $y_i$ because $D$ is free of $F_1$, implying that $|Y|\ge\delta^-+1$, which   contradicts  the hypothesis of $(ii)$, and the proof of $(ii)$ is completed.

Next, to see $(iii)$ let us show that $|X|\ge \delta^-+1$.
To do that, let us see that for each $v_i\in N^-(x)$ one can associate to it a vertex $z_i\in X\setminus\{x\}$ in such a way that $z_i\neq z_j$ for all $i\neq j$. Let us consider the following partition of $N^-(x)$: $N^-(x)\cap (Y\setminus X)$, $N^-(x)\cap X$ and  $N^-(x)\cap (V\setminus {(X\cup Y}))$.
We have the following cases (see Figure \ref{theo:ident-code-digraphs-altern}):

\noindent Case 1: $v_i\in N^-(x)\cap (Y\setminus X)$. Since $\delta^-\ge 2$, there is $w_i\in  N^-(v_i)\setminus\{x\}\subseteq N^-[Y]\setminus\{x\}=N^-[X]\setminus\{x\}$. Hence: If $w_i\in X$, then $z_i=w_i$ and $z_i\neq x$; and if $w_i\not\in X$, since $w_i\in N^-[Y]=N^-[X]$, there exists $z_i\in X$ such that $z_i\in N^+(w_i)$.
     In this case we may assume that $z_i\neq x$, because $D$ is free of $F_1$.

\noindent Case 2: $v_i\in N^-(x)\cap X$. Then $z_i=v_i$ and $z_i\neq x$.

\noindent Case 3: $v_i\in N^-(x)\cap(V\setminus (X\cup Y))\subseteq N^-[X]\setminus (X\cup Y)=N^-[Y]\setminus (X\cup Y)$.  Then we consider the vertices  $y_i\in Y$ such that $y_i\in N^+(v_i)$  and $y_i\ne y_j$ for $i\ne j$. If $y_i\in X$, then $z_i=y_i$, and $y_i\neq x$ because $x\in X\setminus Y$. If $y_i\in Y\setminus X$, then there exists $z_i\in X$ such that $z_i\in N^+(y_i)$. Observe that $z_i$ is different from $x$, because  $D$ is free of $F_1$.
 \begin{figure}[t]
\vskip-.75cm
    \begin{center}
  \includegraphics[width=10cm]{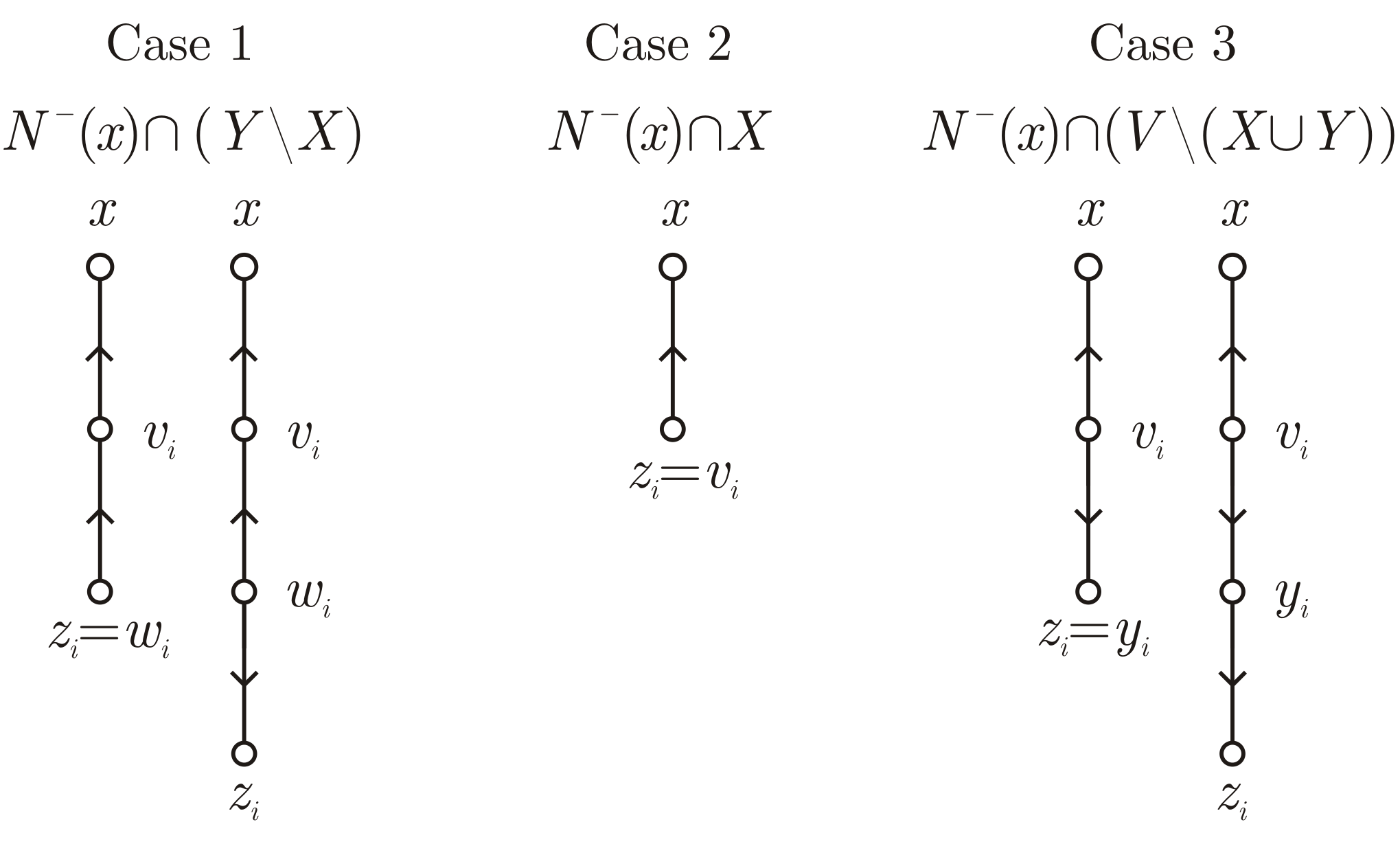}
  \end{center}
  %\vskip-10.25cm
  \caption{All the cases in the proof of $(iii)$ of Theorem~\ref{theo:ident-code-digraphs-altern}.}
  \label{fig:camins}
\end{figure}

Now let us see that all $z_i$ are different. For this, let $i,j\in\{1,\dots,\tau\}$ such that $i\neq j$.
 If $v_i,v_j\in N^-(x)\cap (Y\setminus X)$ and $z_i=z_j$, then (see Figure \ref{fig:camins} Case 1) it could be $w_j=z_j=z_i=w_i\in X$, and  $D$ would contain the subdigraph $F_3$, contradicting the hypothesis of $(iii)$. It could be $z_j=z_i=w_i\in X$ and $w_j\not\in X$, then $D$ would contain the subdigraph $F_5$, a contradiction.  Finally, it could be  $w_i,w_j\not\in X$, $z_i=z_j$ and $z_i\in N^+(w_i)\cap N^+(w_j)$, then $D$ would contain the subdigraph $F_8$, a contradiction.  Therefore,   all the $z_i$ are different in Case 1.
 If $v_i,v_j\in N^-(x)\cap X$ it is clear that $z_i\neq z_j$ in Case 2.
If $v_i,v_j\in N^-(x)\cap(V\setminus (X\cup Y))$ and $z_i=z_j$, then (see Figure \ref{fig:camins} Case 3) it could be $z_j=y_i\in X$, and  $D$ contains the subdigraph $F_6$. Hence, $y_i,y_j\in Y\setminus X$ and $D$ contains  the subdigraph $F_8$. Therefore, all the $z_i$ are different in Case 3.
It remains to prove that for all $i,j\in\{1,\dots,\tau\}$, with $i\ne j$, $z_i\ne z_j$ when $v_i$ and $v_j$ are in different partite subsets of the considered partition of $N^-(x)$.
 Thus, if $z_i=z_j$ for some $i\neq j$, with
 $v_i\in  N^-(x)\cap (Y\setminus X)$ and $v_j\in N^-(x)\cap X$, then $D$ contains one of the subdigraphs $F_1$  or $F_3$ (see Figure \ref{fig:camins} Cases 1 and 2); if $v_i\in  N^-(x)\cap (Y\setminus X)$ and $v_j\in N^-(x)\cap(V\setminus (X\cup Y))$, then $D$ contains one of the subdigraphs  $F_4$, $F_6$, $F_7$ or $F_9$ (see Figure \ref{fig:camins} Cases 1 and 3); and finally, if $v_i\in N^-(x)\cap X$ and $v_j\in N^-(x)\cap(V\setminus (X\cup Y))$, then $D$ contains one of the subdigraphs $F_1$ or $F_4$ (see Figure \ref{fig:camins} Cases 2 and 3).
In any case, we can conclude that $X$   has at least $\delta^-+1$ vertices, which is a contradiction   because $|Y|, |X|\le \delta^-$ in case $(iii)$,   and the proof of this case is completed.

To prove $(iv)$, we assume that $|X|= \delta^-+1$ and $  |Y|\le\delta^-+1$.
 Since by hypothesis $\delta^-\ge 2$, reasoning as in $(iii)$ it follows that $X=\{z_1,z_2,\dots,z_{\delta^-},x\}$ and $d^-(x)=\delta^-$. Hence, by hypothesis $x$ does not  lay on a digon.  Let  $y\in N^+(x)$ with $y\in Y\setminus N^-(x)$.   First, let us show that $y\in Y\cap X$. Suppose that $y\in Y\setminus X$.
 Observe that for all $u\in Y\setminus X$, it can be proved analogously that   $d^-(u)=\delta^-$.
   Since $\delta^-\geq2$, there is $z\in N^-(y)\setminus \{x\}$. Let us show that $z\not\in  X$. Otherwise, suppose $z\in X$, then $z=z_j$ for some $j=1,\ldots, \delta^-$.
 If $v_j\in N^-(x)\cap (Y\setminus X)$, then $D$ contains $F_4$ or $F_5$ (see Figure~\ref{fig:camins} Case 1);
    if $v_j\in N^-(x)\cap X$, then $D$ contains $F_1$ (see Figure~\ref{fig:camins} Case 2);
    and if $v_j\in N^-(x)\setminus (X\cup Y)$, then $D$ contains $F_3$ or $F_5$ (see Figure~\ref{fig:camins} Case 3). Therefore,   $z\notin X$.
     Hence, $z\in N^-(z_i)$ for some $i\in\{1,\dots,\delta^-\}$.
     If $v_i\in N^-(x)\cap (Y\setminus X)$, then $D$ contains $F_7$ or $F_{10}$;
if $v_i\in N^-(x)\cap X$, then $D$ contains $F_4$;
and if $v_i\in N^-(x)\setminus (X\cup Y)$, then $D$ contains $F_6$ or $F_{11}$, a contradiction.
This implies that $y\in X\cap Y$ as we claimed.  So $y=z_i$ for some $i=1,\ldots, \delta^-$.
Notice that if $v_i\notin N^-(x)\cap(Y\setminus X)$, then $x$ would be contained in a digon, or $D$ contains $F_1$ or $F_3$, a contradiction. If $v_i\in Y\setminus X$, then   reasoning for $v_i$ as for $x$, we obtain that every $t\in N^+(v_i)\cap X$ satisfies that $t\in X\cap Y$. However, $x\in N^+(v_i)\cap X$, but $x\not\in Y$, which is a contradiction and the proof of $(iv)$ is done.

 To prove $(v)$ we assume that $\delta^-=1$ and $|X|=2$.
Clearly, the following claims holds if $\delta^-\ge 2$; moreover, since there are no vertices of in-degree 1 laying  on a digon  and by Remark \ref{TT3}, the   claim follows.

\begin{claim}
\label{DigoNo}
Let $(u,v)\in A(D)$. Then, there is $w\in N^-(u)\setminus N^-[v]$.
\end{claim}

First observe that if $|Y|=1$, say $Y=\{y\}$, then $x\in N^-(y)$ and by Claim \ref{DigoNo}, there is $w\in N^-(x)\setminus N^-[y]$, implying that $N^-[X]\neq N^-[Y]$, a contradiction. Then $|Y|=|X|=2$.

Let $X=\{x,x'\}$, $x\in X\setminus Y$,  and $Y=\{y, y'\}$ with $y\in N^+(x)$. Let us prove that the arc $(x,y)$ is not on a digon. Otherwise, suppose that $(x,y), (y,x)\in A(D)$.
%Let us prove that for all $x\in X$ and $y\in Y$, if $(x,y)\in A(D)$, then $(y,x)\notin A(D)$.
By Claim \ref{DigoNo}, there exist $w,z\in V(D)$ such that $z\in N^-(x)\setminus N^-[y]$, and $w\in N^-(y)\setminus N^-[x]$. Hence, $z\in N^-[y']$ and $w\in N^-[x']$.
If $z\notin Y$,  then  $z\ne y' $ and   $z\in N^-(y')$. Moreover, since $D$ is free of $F_1$, $y'\in N^-[x']\subset N^-[X]$. If $x'=y'$, then $w\ne x'$ because $D$ is free of $F_3$;
 $w\in N^-(x')$, implying that $D$ contains $F_6$, therefore $x'\neq y'$ (and so $y'\in Y\setminus X$).
  Moreover, we can assume that $w\notin\{y',x'\}$, otherwise $D$ contains $F_4$ or $F_5$. Thus, $w\in N^-(x')$ implying that $D$ contains $F_{11}$, concluding that If $z\in  Y$. Hence, let us assume that $Y=\{y,z\}$, and analogously  $X=\{x,w\}$.
  By Claim \ref{DigoNo}, there is $u\in (N^-(z)\setminus N^-[x])\cap N^-[w]$, because $N^-[Y]=N^-[X]$, then $D$ contains $F_3$ if $u=w$ or $F_5$ if $u\in N^-(w)$. Therefore,  the arc $(x,y)$ is not on a digon.

Suppose that $X\cap Y\neq\emptyset$. First assume that $X=\{x,x'\}$ and $Y=\{y,x'\}$. Taking into account that $N^-[Y]=N^-[X]$ we have $x\in N^-(y)\cup N^-(x')$ and $y\in N^-(x')$ because $(x,y)$ is not on a digon.
By Claim \ref{DigoNo} there is $w\in N^-(x)\setminus N^-[y]$ and  $w\in  N^-[x']$ (because  $N^-[X]=N^-[Y]$). If $w=x'$, then $(xyx'x)$ is a $3$-cycle in $D$, and by hypothesis there is $u\in N^-(x)\setminus\{x'\}$. By Remark \ref{TT3}, $u\notin N^-(y)\cup N^-(x')$, a contradiction. Then, $w\neq x'$, implying that $D$ contains $F_4$.
Secondly, assume that $X=\{x,y\}$ and $Y=\{y,y'\}$. By Claim \ref{DigoNo} there is $w\in N^-(x)\setminus N^-[y]$ and  $w\in  N^-[y']$. If $w=y'$, there is $w'\in N^-(y')\setminus N^-[x]$ by Claim \ref{DigoNo}, and $D$ would contain a $F_4$. Thus  $w\ne y'$ and $w\in  N^-(y')$, and since $y'\in N^-(x)\cup N^-(y)$ $D$ would contain a $F_1$ or $F_3$, a contradiction.

Suppose that $X\cap Y=\emptyset$. Let $X=\{x,x'\}$ and $Y=\{y,y'\}$. Then,  $y\in N^-(x')$, and since $y\in Y\setminus X$, reasoning for $y$ as for $x$, the arc $(y,x')$ is    like the arc $(x,y)$ and so it is not lying on a digon. Then $x'\in N^-(y')$  and similarly,
$y'\in N^-(x)$. By hypothesis there are no vertices of in-degree $1$  lying   on a 4-cycle, it follows that there is $z\in N^-(x)\setminus \{y'\}$, but by Remark \ref{TT3}, $N^-(x)\cap(N^-(y)\cup N^-(y'))=\emptyset$ implying that $N^-[X]\neq N^-[Y]$, a contradiction.
\end{proof}

%Observe that Kautz digraph $K(2,2)$ is free of $TT_3$ but it contains subdigraphs of type $(b)$. And the circulant digraph of Example \ref{circu} is not free of $TT_3$.  In this sense Theorem \ref{theo:ident-code-digraphs-altern} is best possible.

If for each graph $G$, we consider its corresponding symmetric digraph $ \overset{\text{$\leftrightarrow $}}{G}$, obtained by replacing each edge $uv\in G$ by the arcs $(u,v)$ and $(v,u)$, then we obtain the following corollary from Theorem \ref{theo:ident-code-digraphs-altern}.
\begin{corollary}\label{graph} Let $G$ be a graph of girth $g$ and minimum degree $\delta\ge 2$. Then
\begin{enumerate}
\item If  $g\geq7$, then $G$ admits a $(1,\leq \delta)$-identifying code.
\item If  $g\geq5$, then $G$ admits a $(1,\leq \delta-1)$-identifying code.
\end{enumerate}
\end{corollary}

Observe that Theorem~\ref{laiho} by Laihonen is a consequence of Corollary~\ref{graph}.

%\begin{figure}[t]
%\vskip-.75cm
%    \begin{center}
%  \includegraphics[width=10cm]{Conf-prohibides8bis.pdf}
%  \end{center}
%  \vskip-13cm
%  \caption{All the forbidden subdigraphs of Theorem~\ref{theo:ident-code-digraphs-altern} $(iii)$.}
%  \label{ProhibidesDelta}
%\end{figure}

\section{1-in-regular  and 2-in-regular digraphs}

In this section, we characterize  the  $d$-in-regular digraphs admitting a $(1,\leq d)$-identifying code and a  $(1,\leq d+1)$-identifying code for $d=1,2$. Recall that by Proposition \ref{BoundDelta+1}, if $D$ is a $d$-in-regular digraph admiting a $(1,\leq\ell)$-identifying code, then $\ell\leq d+1$. We start by giving a characterization of  1-in-regular digraphs admitting a $(1,\leq 2)$-identifying code.
 Observe that every  1-in-regular digraph $D$ admits a $(1,\leq 1)$-identifying code if and only if $D$ does not contain digons.

\begin{theorem}
\label{prop-cicle} Every  1-in-regular digraph $D$ admits a $(1,\leq 2)$-identifying code if and only if the girth of $D$ is at least 5.
\end{theorem}
\begin{proof}  Let $(u_1u_2u_3u_1)$  be  a directed triangle or $(v_1v_2v_3v_4v_1)$ a 4-cycle in $D$, then the sets $X_1=\{u_1,u_3\}$, $Y_1=\{u_2,u_3\}$, $X_2=\{v_1,v_3\}$ and $Y_2=\{v_2,v_4\}$ are such that $N^-[X_i]= N^-[Y_i]$, for $i=1,2$. Therefore, if $D$ contains a $k$-cycle, for some $k=2,3$ or $4$, then $D$ does not admit a $(1,\leq2)$-identifying code.
Conversely, suppose that the girth of $D$ is at least 5. Since $D$ is $1$-in-regular it follows that $D$ does not contain  any  subdigraph isomorphic to $F_1$, $F_3$, $F_4$, $F_5$, $F_6$ nor $F_{11}$ of Figure~\ref{fig:conf-prohibides3}, then by Theorem~\ref{theo:ident-code-digraphs-altern}, $D$ admits a $(1,\leq  2)$-identifying code. This completes the proof.
\end{proof}

%%%%%%%%%%%%%%%%%%%%%%%%%%%%%%%%%%%%%%%%%%%%%%%%%%%%%%%%%%%%%%%%%%%%%%%%%%%%%%%%%%%%%%%%%%%%%%%%%%%%%%%%%%%%%%%%%%%%%%%%%%%%%%%%%%%%%%%%%%%%%%%
%%%%%%%%%%%%%%%%%%%%%%%%%%%%%%%%%%%%%%%%%%%%%%%%%%%%%%%%%%%%%%%%%%%%%%%%%%%%%%%%%%%%%%%%%%%%%%%%%%%%%%%%%%%%%%%%%%%%%

%\begin{figure}[t]
%\vskip-.75cm
%    \begin{center}
%  \includegraphics[width=14cm]{exemple-no-codi-iden.pdf}
%  \end{center}
%  \vskip-15.25cm
%  \caption{$(a)$ The Kautz digraph $K(2,2)$ does not admit a $(1,\leq 2)$-identifying code.}
%  \label{fig:exemple}
%\end{figure}
%
%\begin{figure}[t]
%    \begin{center}
%  \includegraphics[width=14cm]{doble-corel.pdf}
%  \vskip-15.25cm
%  \caption{Two 2-regular  digraphs on 6 vertices without a $(1,\leq 2)$-identifying code\textcolor[rgb]{0.00,0.50,1.00}.}
%  \label{doble}
%  \end{center}
%\end{figure}

The following result gives a complete characterization of all 2-in-regular digraphs admitting a $(1,\leq 1)$-identifying code and a characterization of all 2-in-regular digraphs admitting a $(1,\leq 2)$-identifying code.

{\begin{theorem} \label{char2} Let $D$ be a 2-in-regular %connected
digraph.
 \begin{itemize}
   \item[$(i)$] $D$ admits a $(1,\leq 1)$-identifying code if and only if it does not contain any subdigraph isomorphic to $H_1$ of Figure \ref{fig:conf-prohibides6}.
   \item[$(ii)$] $D$ admits a $(1,\leq 2)$-identifying code if and only if it does not contain any subdigraph isomorphic to one of the digraphs of Figure \ref{fig:conf-prohibides6}.
 \end{itemize}
\end{theorem}
\begin{proof}
 In what follows, for brevity, we made reference to the different cases $H_1$-$H_{13}$ of Figure \ref{fig:conf-prohibides6} without mentioning the figure.   First note that any digraph with twins and minimum in-degree at least 2, necessarily contains $H_1$. Hence, the proof of $(i)$ follows by Remark \ref{twin}, because the vertices $x,y$ of $H_1$   are twins.  To prove $(ii)$, first let $X=\{x,x'\}$ (or $X=\{x\}$) and $Y=\{y, y'\}$. It is direct to check that $N^-[X]=N^-[Y]$ in each one of the digraphs shown in Figure \ref{fig:conf-prohibides6}. For the converse, we assume that $D$ does not contain any subdigraph isomorphic to the digraphs depicted in Figure \ref{fig:conf-prohibides6}, and $N^-[X]=N^-[Y]$ for   $X\ne Y$ such that $1\le |Y|\le |X|\le 2$.
According to $(i)$, $|X|=2$, consequently $3\le |N^-[X]|\le 6$. Notice that if $|Y|=1$, then $|N^-[Y]|=3$, and so $|N^-[X]|=3$ yielding that $D$ contains $H_1$. Therefore, we assume that $|Y|=|X|= 2$. Let $X=\{x,x'\}$ and $Y=\{y,y'\}$ with $x\in X\setminus Y$. Let $N^-(x)=\{v_1,v_2\}$ and $y\in Y$ such that $y\in N^+(x)$. As we did in the proof of Theorem \ref{theo:ident-code-digraphs-altern} we consider the different cases according to the partition of $N^-(x)$: $N^-(x)\cap (Y\setminus X)$, $N^-(x)\cap X$ and $N^-(x)\cap (V\setminus {(X\cup Y}))$.
\begin{figure}[t]
\vskip-.75cm
    \begin{center}
  \includegraphics[width=14cm]{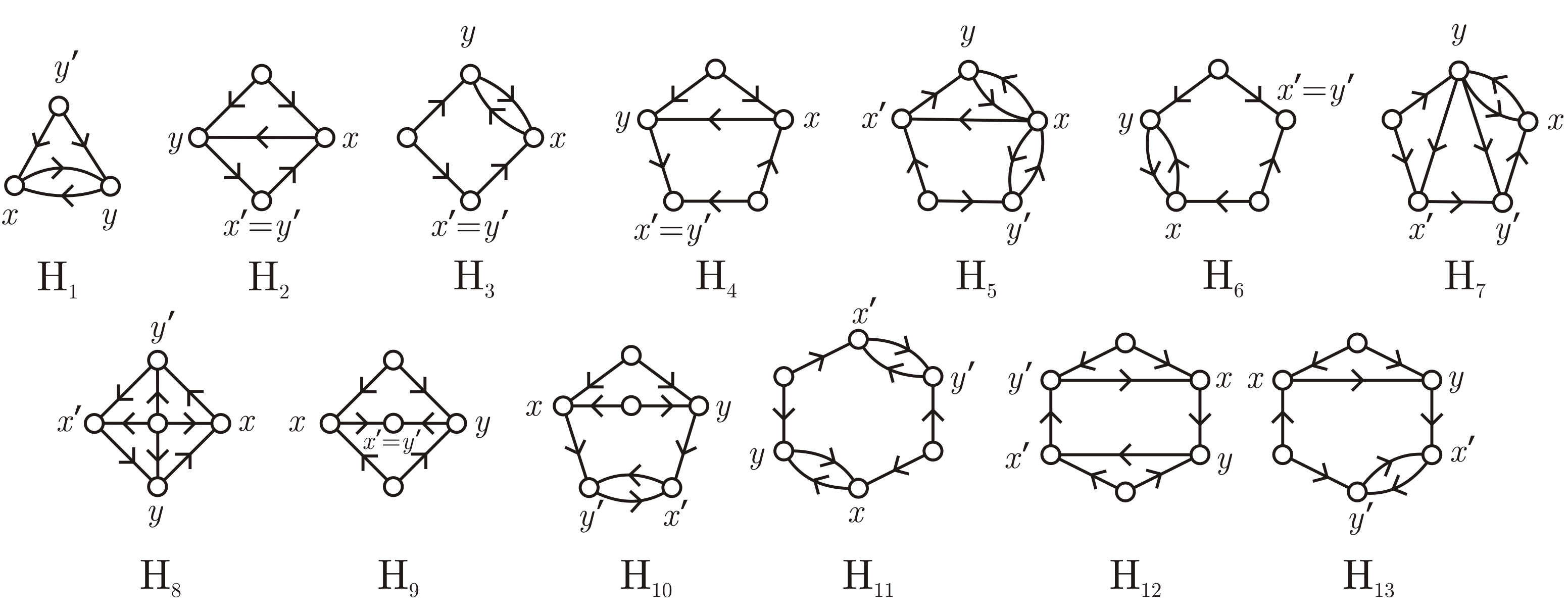}
  \end{center}
 % \vskip-6cm
  \caption{The forbidden subdigraphs in a 2-in-regular digraph admitting a $(1,\le 2)$-identifying code.}
  \label{fig:conf-prohibides6}
\end{figure}

\noindent Case 1: Suppose that $v_1,v_2\in Y\setminus X$. Let $y=v_1$ and $y'=v_2$ and observe that in this case $x'\notin Y$. As $D$ is $H_1$-free and $H_3$-free, $(N^-(y)\setminus\{x\})\cap N^-[y']=\emptyset$  and  there is no arc between $y'$ and    $  N^-(y)\setminus\{x\}$. Let $w\in N^-(y)\setminus\{x\}$ and $w'\in N^-(y')\setminus\{x\}$, then $w,w'\in N^-[x']$.

Subcase 1.1: Suppose that $\{w,w'\}\cap\{x'\}=\emptyset$. Hence, $N^-(x')=\{w,w'\}$. Since $x'\in N^-[Y]$ it follows that   $x'\in N^-(y')$  implying that $D$ contains $H_{13}$, a contradiction.

Subcase 1.2: Suppose that $x'=w$. Hence, $w'\in N^-(x')$. If there is $z\in N^-(x')\setminus(X\cup Y\cup\{w'\})$, then $z\in N^-(y')$, implying that $D$ contains $H_{10}$, a contradiction. Therefore, $N^-[X]=X\cup Y\cup\{w'\}$, implying that $N^-(x')=\{w',x\}$ or $N^-(x')=\{w',y\}$. First suppose that $N^-(x')=\{w',x\}$. If $x\in N^-(y')$ then $D$ contains $H_5$ and if $y\in N^-(y')$, then $D$ contains $H_4$, a contradiction. Therefore, $N^-(x')=\{w',y\}$. If $x\in N^-(y')$ then $D$ contains $H_6$ and if $y\in N^-(y')$, then $D$ contains $H_5$, a contradiction.

Subcase 1.3: Supposse that $x'=w'$. Hence, $w\in N^-(x')$. If there is $z\in N^-(x')\setminus(X\cup Y\cup\{w'\})$, then $z\in N^-(y')$, implying that $D$ contains $H_{13}$, a contradiction. Therefore, $N^-[X]=X\cup Y\cup\{w'\}$. Hence, $N^-(x')=\{w,x\}$ or $N^-(x')=\{w,y\}$. First suppose that $N^-(x')=\{w,x\}$. If $x\in N^-(y')$ then $D$ contains $H_4$ and if $y\in N^-(y')$, then $D$ contains $H_9$, a contradiction. Therefore, $N^-(x')=\{w,y\}$. Hence, if $x\in N^-(y')$ then $D$ contains $H_4$ and if $y\in N^-(y')$, then $D$ contains $H_7$, a contradiction.

\noindent Case 2: Suppose that $v_1,v_2\in X$. Since $|X|=2$ this case is not possible.

\noindent Case 3: Suppose that $v_1,v_2\notin (X\cup Y)$. Since $x\in N^-(y)$, then $|N^-(y)\cap\{v_1,v_2\}|\leq1$ implying that $\{v_1,v_2\}\cap N^-(y')\neq\emptyset$. Without loss of generality suppose that $v_1\in N^-(y')$.

Subcase 3.1: If $y\in Y\setminus X$, then $y\in N^-(x')$. If $y'\in X\cap Y$, i.e. $y'=x'$, then $v_2\in N^-(y)$, implying that $D$ contains $H_4$. If $y'\in Y\setminus X$, then $N^-(x')=\{y,y'\}$ and $x'\in N^-(y)\cup N^-(y')$. If $x'\in N^-(y)$, then $v_2\in N^-(y')$, implying that $D$ contains $H_{10}$. And, if $x'\in N^-(y')$, then $v_2\in N^-(y)$, implying that $D$ contains $H_{13}$.

Subcase 3.2: If $y\in X\cap Y$ i.e. $x'=y$, then $y'\in N^-(y)$ and $v_1,v_2\in N^-(y')$, hence $D$ contains $H_9$, a contradiction. Therefore, the proof of Case 3 is finished.

\noindent Case 4: Suppose that $v_1\in Y\setminus X$ and $v_2\in X$, that is, $v_2=x'$.
Observe that if $v_1\in N^+(x)$, since $D$ is $H_1$-free, there is $w\in V(D)\setminus X$ such that $w\in N^-(v_1)\subset N^-[ Y] $. Thus,  $w\in N^-(x')$, implying that $D$ contains $H_3$, a contradiction. Then $v_1\not\in N^+(x)$ and so $v_1=y'$, and moreover $y\in N^-(x')$. If $x'\in N^+(x)$, then
  $N^-[X]=\{x,x',y,y'\}$, yielding that $y\in N^-(y')$, contradicting that $D$ is $H_3$-free.
Therefore, $N^+(x)\cap\{y',x'\}=\emptyset$ and recall that $y\in N^-(x')$.  Moreover, reasoning for $y$ as for $x$ in Case 1, we get that $x'\notin N^-(y)$. Moreover, if $y'\in N^-(y)$, then $D$ contains $H_2$, a contradiction. Therefore, there is $w\in N^-(y)\setminus (X\cup Y)$. Hence, $w\in N^-(x')$, implying that $D$ contains $H_2$, a contradiction.

\noindent Case 5: Suppose that $v_1\in Y\setminus X$ and $v_2\notin(X\cup Y)$.

Subcase 5.1: Suppose that $v_1\in  N^+(x)$, then, we can assume that $v_1=y$. Since $D$ is $H_1$-free, $v_2\in N^-(y')$ and there is $w\in V(D)\setminus\{x,v_2\}$ such that $N^-(y)=\{x,w\}$.   Observe that since $D$ is $H_3$-free, $v_2\notin N^-(w)$, then $w\neq y'$. Moreover, since $D$ is $H_6$-free, $w\notin N^-(y')$.  Hence, $w\in N^-[x']$, implying that $x'\neq y'$. Observe that reasoning for $y$ as for $x$ in Case 1, we get that $w\neq x'$. Then, $w\in N^-(x')$ and, since $x',y'\in N^-[X]=N^-[Y]$, it follows that $x'\in N^-(y')$ and $y'\in N^-(x')$, therefore $D$ contains $H_{11}$, a contradiction.

Subcase 5.2: Suppose that $v_1\notin   N^+(x)$, then $v_1=y'$ and $y\in N^-[x']$. First suppose that  $y=x'$. If $N^-(y')\subseteq X\cup\{v_2\}$, then  $N^-(y')=\{x',v_2\}$ implying that $D$ contains $H_2$. Hence, there is $w\in N^-(y')\setminus (X\cup\{v_2\})$. Then, $w\in N^-(x')$ and $v_2\in N^-(y')$, implying that $D$ contains $H_4$, a contradiction. Therefore, $y\ne x'$, implying that $y\in N^-(x')$. Reasoning for $y$ as for $x$ in Case 1 and Case 4 it follows that $N^-(y)\cap\{x',y'\}=\emptyset$. Then, $x'\in N^-(y')$. Moreover, since $v_2\in N^-(x)$, $v_2\in N^-(y)\cup N^-(y')$.
Also, reasoning for $x'$ as for $x$ in Case 1 and Case 4 it follows that $ N^-(x')\cap\{x,y'\}=\emptyset$. Hence, if $v_2\in N^-(y)\cap N^-(y')$, then $N^-[Y]=X\cup Y\cup\{v_2\}$, implying that $v_2\in N^-(x')$. Then, $D$ contains $H_{8}$, a contradiction. If $v_2\in N^-(y')\setminus N^-(y)$, then there is $z\in N^-(y)\setminus (X\cup Y\cup\{v_2\})$, implying that $N^-(x')=\{y,z\}$ and $D$ contains $H_{12}$. Analogously if $v_2\in N^-(y)\setminus N^-(y')$. And the proof of this case is completed.

\noindent Case 6: Suppose that $v_1\in X$ and $v_2\notin (X\cup Y)$. That is, $v_1=x'$. If $x'\in X\setminus Y$, then $y\in N^-(x')$. Since $y\in Y\setminus X$, reasoning for $x'$ as for $x$ in  Case 1, 4 and 5, we reach a contradiction. Hence, $x'\in X\cap Y$. If $x'=y$, then $y'\in N^-(x')$ and $v_2\in N^-(y')$, implying that $D$ contains $H_3$. Therefore, $x'\neq y$ and hence, $y\in Y\setminus X$. Since $x\in N^-(y)$, reasoning for $y$ as for $x$ in  Case 1, 4 and 5, we reach a contradiction.
\end{proof}

\begin{corollary}\label{coro} Every $TT_3$-free 2-in-regular oriented graph admits a $(1,\leq 2)$-identifying code if and only if it does not contain any subdigraph isomorphic to $H_9$ of Figure \ref{fig:conf-prohibides6}.

\end{corollary}

Observe that Corollary \ref{coro} is an improvement  of Theorem \ref{theo:ident-code-digraphs-altern} $(ii)$ for 2-in-regular oriented digraphs. Now, the $TT_3$-free and 2-in-regular oriented graph  can have two distinct vertices $u,v$  with $|N^-(u)\cap N^-(v)|=2$,  that is, a subdigraph $F_2$ of Figure \ref{fig:conf-prohibides3}, but in this case there is no vertex $w\in V$ such that $u,v\in N^-(w)$.

In the following theorem we characterize the $2$-in-regular digraphs admitting a $(1,\le 3)$-identifying code.

\begin{theorem} \label{tres} Let $D$ be a 2-in-regular
digraph. Then  $D$ has a $(1,\leq 3)$-identifying code if and only if it   is a $TT_3$-free oriented graph, and does not contain any subdigraph isomorphic to one of the digraphs of Figure \ref{fig:conf-prohibides33}.
\end{theorem}

\begin{figure}[t]
%\vskip-.75cm
    \begin{center}
  \includegraphics[width=14cm]{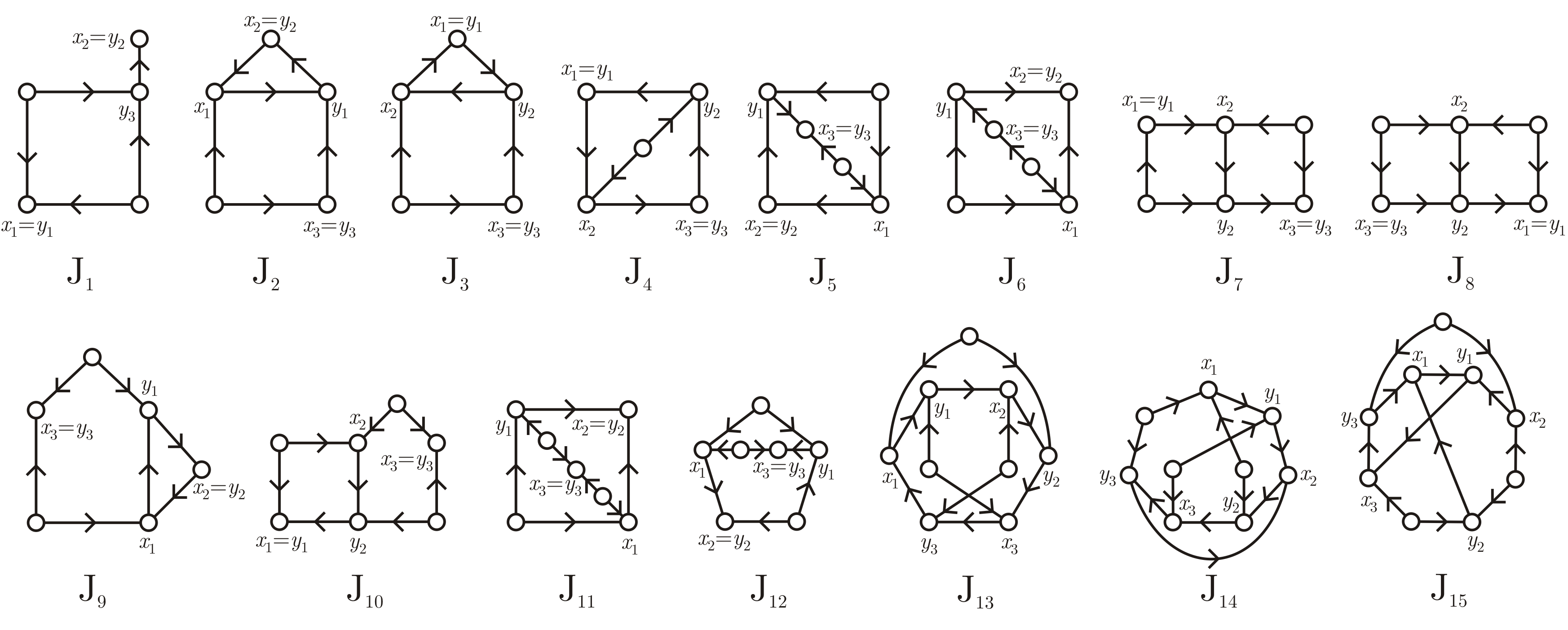}
  \end{center}
 % \vskip-5cm
  \caption{All the forbidden subdigraphs of Theorem~\ref{tres}.}
  \label{fig:conf-prohibides33}
\end{figure}

\begin{proof}
By Proposition \ref{BoundDelta+1}, if $D$ contains a digon, then $D$ does not admit a $(1,\leq3)$-identifying code. Suppose that $D$ contains a $TT_3$, let say $w\in N^-(u)\cap N^-(v)$ and $(u,v)\in A(D)$, and let $z\in V(D)$ such that $N^-(u)=\{w,z\}$. Hence, the sets $X=\{z,u,v\}$ and $Y=\{z,v\}$ has the same closed in-neiborhood. Furthermore, for every digraph shown in Figure \ref{fig:conf-prohibides33} let $X=\{x_1,x_2,x_3\}$ (or $X=\{x_1,x_2 \}$) and $Y=\{y_1, y_2,y_3\}$. It is direct to check that $N^-[X]=N^-[Y]$ in each case. To the converse, we reason by contradiction. Let $D$ be a $TT_3$-free oriented graph without the  subdigraphs of Figure \ref{fig:conf-prohibides33}. Let $X,Y\subseteq V(D)$, $X\ne Y$,  with $N^-[X]=N^-[Y]$ and such that $1\leq|X|\leq|Y|\leq3$.
Since $D$ does not contain a subdigraph isomorphic to  $J_1$ of Figure \ref{fig:conf-prohibides33}, then it does not contain a  subdigraph $H_9$ of Figure \ref{fig:conf-prohibides6}. By Corollary \ref{coro}, $D$ admits a $(1,\leq 2)$-identifying code. Hence, $|Y|=3$, $|N^-[Y]|\ge 6$ and $|X|\ge 2$. In what follows, for brevity, we always make reference to the different cases  $J_1$-$J_{15}$ of Figure \ref{fig:conf-prohibides33} without mentioning the figure. Let us prove the following claim.
\begin{claim}
 \label{A}
Let $a,b\in V(D)$, with $a\neq b$, be such that $N^-(a)\subseteq N^-[b]$. Then, $N^-(a)=N^-(b)$ and  $N^+(a)=N^+(b)=\emptyset$.
\end{claim}
\begin{proof}
If $b\in N^-(a)$, then   $D$ contains a $TT_3$, which is a contradiction. Hence, $N^-(a)=N^-(b)$ and  $N^+(a)=N^+(b)=\emptyset$, because otherwise  $D$ contains $J_1$.
\end{proof}

Suppose $X=\{x_1,x_2\}$, then $|N^-[X]|=6$ (because $N^-[X]=N^-[Y]$) and $N^-[x_1]\cap N^-[x_2] =\emptyset$.
Let $ N^-(x_1)=\{u,v\}$ and $ N^-(x_2)=\{z,t\}$, so that $N^-[X]=\{x_1,x_2,u,v,z,t\} = N^-[Y]$.
Without loss of generality, we may assume that $u\in Y$. Since $D$ has neither digon nor $TT_3$, $N^-(u)\subseteq N^-[x_2]$, which implies by Claim \ref{A} that  $N^-(u)= N^-(x_2) $ and $N^+(u)=\emptyset$, a contradiction.
Therefore, $|X|=|Y|=3$. Let us denote $X=\{x_1,x_2,x_3\}$. We  prove the following claims.

\begin{claim}
 \label{B}
Let $a,b,c\in V(D)$. If $N^-[a]\subseteq N^-[b]\cup N^-[c]$,  then $a\in\{b,c\}$.
\end{claim}

\begin{proof}
If $a\notin\{b,c\}$, then without loss of generality let us assume that $a\in N^-(b)$. Hence, by Remark \ref{TT3}, $N^-(a)\subseteq N^-[c]$, which contradicts Claim \ref{A} because $N^+(a)\ne \emptyset$.
\end{proof}

\begin{claim}
 \label{Dif in-neigh of X}
 $N^-(x_i)\neq N^-(x_j)$ for all $1\leq i<j\leq3$.
\end{claim}

\begin{proof}
 Suppose that $N^-(x_1)=N^-(x_2)$. Then, $N^+(x_1)=N^+(x_2)=\emptyset$, because $D$ is $J_1$-free, which implies $x_1,x_2\in Y$. Since $|N^-[X]|\geq 6$, there is $z\in N^-(x_3)\setminus \left(N^-[x_1]\cup N^-[x_2]\right)$. Because $\{x_3,z\}\subseteq N^-[Y]$, $D$ must contain a digon if $z=y_3\in Y$,  or a $TT_3$ if $\{x_3,z\}=N^-(y_3)$, which is a contradiction. Therefore, $N^-(x_1)\neq N^-(x_2)$.
\end{proof}

\begin{claim}
 \label{comun}
If $7\leq|N^-[X]|\leq8$, $N^-(x_i)\cap N^-(x_j)=\{v\}$, $i\ne j$, and there are exactly two or  no arc  between the elements of $X$,   then $|Y\cap \{x_i,x_j\}|\le 1$.
\end{claim}
\begin{proof}
We proceed by contradiction. Assume $Y=\{x_1,x_2,y\}$. First suppose that there  is no arc   between the elements of $X$. If $v\in N^-(x_1)\cap N^-(x_2)\cap N^-(x_3)$, then according to Claim \ref{Dif in-neigh of X} $|N^-[X]|=7$ and {$N^-[x_3]\subseteq N^-[x_1]\cup N^-[y]$}, which contradicts Claim \ref{B}. Hence, $N^-(x_1)\cap N^-(x_2)\cap N^-(x_3)=\emptyset$. If $|N^-[X]|=7$, let $N^-(x_1)=\{u,v\}$, $N^-(x_2)=\{v,z\}$ and $N^-(x_3)=\{z,w\}$. Since $N^-(v)\cap N^-[X]\subseteq\{x_3,w\}$, by Remark \ref{TT3}, $v\notin Y$,  and  analogously $z\notin Y$. Consequently, $N^-[x_3]\subseteq N^-[x_2]\cup N^-[y]$, which contradicts Claim \ref{B}.
If $|N^-[X]|=8$, then  $N^-(x_3)\subseteq  N^-[y]$, a contradiction to Claim \ref{A} because $y\not\in X$ and so $N^+(y)\ne \emptyset$.
Finally assume that there are two arcs between the elements of $X$. Notice that by Remark \ref{TT3},  both   arcs between the elements of $X$ are incident in $x_3$. Furthermore, since $7\leq |N^-[X]|\leq 8$ and $N^-(x_1)\cap N^-(x_2)=\{v\}$, $v=x_3$ and $|N^-[X]|=7$, we have $N^-(x_3)\subseteq N^-[y]$, a contradiction to Claim \ref{A}.
\end{proof}

Let $N^-(x_1)=\{u,v\}$.  We distinguish the following cases according to the number of arcs between the vertices of $X$.

\noindent Case 1: First let us assume that there are no arcs between the elements of $X$.

\noindent Subcase 1.1: Suppose $|N^-[X]|=6$. Then, $N^-[X]=\{x_1,x_2,x_3,u,v,z\}$, so
  Claim \ref{Dif in-neigh of X} implies that $|N^-(x_i)\cap N^-(x_j)|=1$ for all $i\neq j$. Let $N^-(x_2)=\{v,z\}$. Observe that $v\not\in N^-(x_3)$, otherwise $N^-(x_3)=N^-(x_i)$ for some $i\in \{1,2\}$,
 contradicting   Claim \ref{Dif in-neigh of X}.
   Therefore $N^-(x_3)=\{u,z\}$.
  Let $y\in Y\setminus X$, then $y\in \{u,v,z\}$.  We can check that $|N^-(y)\cap N^-[X]|\le 1$ for all $y\in \{u,v,z\}$, because $D$ is a $TT_3$-free oriented graph, which is a contradiction.

\noindent Subcase 1.2: Suppose $|N^-[X]|=7$. Then $N^-[X]=\{x_1,x_2,x_3,u,v,z,w\}$. By Claim \ref{Dif in-neigh of X}, there are  two cases to be considered, namely,  $|N^-(x_1)\cap N^-(x_2)\cap N^-(x_3)|=1$ and  $|N^-(x_1)\cap N^-(x_2)\cap N^-(x_3)|=0$.

\noindent Subsubcase 1.2.1: If $|N^-(x_1)\cap N^-(x_2)\cap N^-(x_3)|=1$, w.l.o.g.  $N^-(x_2)=\{v,z\}$ and $N^-(x_3)=\{v,w\}$. Since $D$ is an oriented graph and does not contain $TT_3$, $N^-(v)\cap N^-[X]=\emptyset$, which means that $v\notin Y$ and $v\in N^-(Y)$.  Since
 $N^+(v)\cap\{u,z,w\}=\emptyset$, it follows that $Y\cap X\neq\emptyset$. By Claim \ref{comun},
 $|X\cap Y|= 1$.
 W.l.o.g.  suppose that $X\cap Y=\{x_1\}$. If $Y=\{x_1,z,w\}$, then $x_2\in N^-(w)$ and $x_3\in N^-(z)$, implying that $D$ contains $J_4$.  If $Y=\{x_1,u,z\}$, then $x_2\in N^-(u)$ and $N^-(u)\subseteq\{x_2,x_3,w\}$.
 If $N^-(u)=\{x_2,x_3\}$, then $w\in N^-(z)$ and hence $D$ contains $J_6$. If $N^-(u)=\{x_2,w\}$, then $x_3\in N^-(z)$, which implies that $D$ contains $J_5$.

 \noindent Subsubcase 1.2.2: If  $|N^-(x_1)\cap N^-(x_2)\cap N^-(x_3)|=0$, w.l.o.g.   $N^-(x_2)=\{v,z\}$ and $N^-(x_3)=\{z,w\}$. By Claim \ref{comun}, $|Y\cap \{x_1,x_2\}|\le 1$ and $|Y\cap \{x_2,x_3\}|\le 1$.
 Moreover, if $\{x_1,x_3\}\subseteq Y$, then since $x_2\in N^-[Y]$, we have $\{u,w\}\cap Y\neq\emptyset$;  w.l.o.g. let us assume that $Y=\{x_1,x_3,u\}$. Then, $x_2\in N^-(u)$ and $N^-(u)\subseteq\{x_2,x_3,w\}$. If $N^-(u)=\{x_2,x_3\}$, then $D$ contains $J_8$, and if $N^-(u)=\{x_2,w\}$, then $D$ contains $J_{10}$. Therefore, $|Y\cap X|\leq1$.
Suppose that  $X\cap Y=\{x_1\}$ and let $Y=\{x_1, y,y'\}$,  then $N^-[x_3]\subseteq N^-[y]\cup N^-[y']$, which contradicts Claim \ref{B}. Hence, $X\cap Y\neq\{x_1\}$, and similarly  $X\cap Y\neq\{x_3\}$.
Then, $X\cap Y=\{x_2\}$. If $v\in Y$, then there is $y\in Y\setminus \{x_2,v\}$, such that $N^-(y)=\{x_1,u\}$ contradicting Remark \ref{TT3}. Hence, $v\not\in Y$, and analogously $z\notin Y$. Therefore $Y=\{x_2,u,w\}$, and then $x_1\in N^-(w)$, implying that $N^-(u)=\{x_3,x_2\}$. Consequently, $D$ contains $J_8$.
If $|Y\cap X|=0$, by symmetry, we only have to consider the following two cases. If $Y=\{u,v,z\}$, then $x_2\in N^-(u)$ and $x_1\in N^-(z)$, implying that $D$ contains $J_4$. If $Y=\{u,z,w\}$, then $x_3\in N^-(u)$ implying that $N^-(u)=\{x_2,x_3\}$, and $D$ contains $J_8$.

\noindent Subcase 1.3: Suppose $|N^-[X]|=8$.  W.l.o.g. $N^-(x_2)=\{v,z\}$, and $N^-(x_3)=\{t,w\}$.
Observe that  $v\notin Y$,  otherwise $N^-(v) \subseteq N^-[x_3]$ in contradiction to Claim \ref{A}.
If $Y\cap X=\emptyset$,  then   we can assume that $t\in Y$ and $v\in N^-(t)$.  Consequently, $\{u,z\}\cap N^-(t)=\emptyset$, otherwise $D$ contains $J_1$, therefore $\{x_3,w\}\cap N^-(t)\neq\emptyset$, a contradiction. Therefore $Y\cap X\neq\emptyset$. If $|Y\cap X|=2$, then by Claim \ref{comun}, $\{x_1,x_3\}\subseteq Y$ or $\{x_2,x_3\}\subseteq Y$.
If $Y=\{y,x_2,x_3\}$, then $N^-[x_1]\subseteq N^-[x_2]\cup N^-[y]$, contradicting Claim \ref{B}. Then, $Y\ne \{y,x_2,x_3\}$, and similarly $Y\ne \{y,x_1,x_3\}$.
 Thus $|Y\cap X|=1$.
If $Y=\{x_1,y,y'\}$ or $Y=\{x_3,y,y'\}$, then $N^-[x_3]\subseteq N^-[y]\cup N^-[y']$ or $N^-[x_1]\subseteq N^-[y]\cup N^-[y']$, respectively, which contradicts Claim \ref{B}.

\noindent Subcase 1.4: Suppose $|N^-[X]|=9$. Hence, the in-neighborhoods of the elements of $X$  must be disjoint, the same  is true  for $Y$.
 Let $N^-(x_i)=\{u_i,v_i\}$, for   $i=1,2,3$.  Observe that if $1\leq|X\cap Y|\leq2$, then $N^-[x_i]\subseteq N^-[y]\cup N^-[y']$ for some $i\in \{1,2,3\}$ and  $y,y'\in Y\setminus \{x_i\}$, in contradiction to Claim \ref{B}. Therefore, $X\cap Y=\emptyset$.
 Without loss of generality there are two cases to be considered.

\noindent Subsubcase 1.4.1: If $Y=\{u_1,v_1,u_2\}$, then $x_1\in N^-(u_2)$.
 If $x_3\in N^-(u_1)$, then without loss of generality $u_3\in N^-(v_1)$ and $v_3\in N^-(u_2)$; moreover, $x_2\in N^-(v_1)$ and $v_2\in N^-(u_1)$ or $x_2\in N^-(u_1)$ and $v_2\in N^-(v_1)$, implying that $D$ contains $J_{14}$ or $J_{15}$, respectively.
 If $x_3\in N^-(u_2)$, then we may assume that $u_3\in N^-(u_1)$ and $v_3\in N^-(v_1)$, and  so $x_2\in N^-(u_1)$ and $v_2\in N^-(v_1)$, implying that $D$ contains $J_{15}$.

\noindent Subsubcase 1.4.2: Let $Y=\{u_1,u_2,u_3\}$. Without loss of generality, suppose $x_2\in N^-(u_1)$, then by Remark \ref{TT3}, $N^-(u_1)\setminus \{x_2\}\subseteq N^-[x_3]$. Since there is no arc between the elements of $Y$ there are two cases to be considered.

  1.4.2.1: If $N^-(u_1)=\{x_2,x_3\}$, then $v_3\in N^-(u_2)$ and $v_2\in N^-(u_3)$. Hence, $x_1\in N^-(u_2)$ and $v_1\in N^-(u_3)$, or $v_1\in N^-(u_2)$ and $x_1\in N^-(u_3)$; in any case  $D$ contains $J_{14}$.

  1.4.2.2: If $N^-(u_1)=\{x_2,v_3\}$, then $x_3\in N^-(u_2)$, and $v_2\in N^-(u_3)$. If $x_1\in N^-(u_2)$, then $v_1\in N^-(u_3)$, implying that $D$ contains $J_{14}$. Finally, if $x_1\in N^-(u_3)$, then $v_1\in N^-(u_2)$, implying that $D$ contains $J_{13}$.

\noindent Case 2:  Suppose there is just one arc between the elements of $X$, say  $(x_1,x_2)\in A(D)$. Then, $|N^-(X)|=6,7,8$, and $N^-(x_1)\cap N^-(x_2)=\emptyset$ by Remark \ref{TT3}.  Let   $N^-(x_2)=\{x_1,z\}$ and let us distinguish the following cases.

 \noindent Subcase 2.1: $|N^-[X]|=6$. Hence, $N^-[X]=\{x_1,x_2,x_3,u,v,z\}$,
          and by Claim \ref{Dif in-neigh of X} let us assume  w.l.o.g. that $N^-(x_3)=\{v,z\}$.
         Moreover, since $D$ is an oriented graph and does not contain $J_1$, $N^-(z)\cap N^-[X]\subseteq\{u\}$ and $N^-(v)\cap N^-[X]\subseteq\{x_2\}$, therefore $z,v\notin Y$; hence $u\in Y$. Since $D$ is a $TT_3$-free oriented graph, $N^-(u)\subseteq \{x_2,x_3,z\}$. Moreover, by Remark \ref{TT3} $z\notin N^-(u)$. Hence, $N^-(u)=\{x_2,x_3\}$, implying that $D$ contains $J_2$.

\noindent Subcase 2.2: $|N^-[X]|=7$. In this case, there is $w\in N^-(x_3)\setminus \left(X\cup\{u,v,z\}\right)$. By symmetry, $N^-(x_3)=\{z,w\}$ or $N^-(x_3)=\{v,w\}$.
          		 First suppose that $N^-(x_3)=\{z,w\}$. Since $D$ is a $TT_3$-free oriented graph,  if $z\in Y$, then $N^-(z)=N^-(x_1)$, which is a contradiction by Claim \ref{A}. Hence $z\not\in Y$. Analogously,
            if $w\in Y$ and $x_2\in N^-(w)$, then $N^-(w)\subseteq\{x_2,u,v\}$, implying that $D$ contains $J_7$; and if $x_2\not\in N^-(w)$, then $N^-(w)\subseteq N^-[x_1]$, contradicting Claim \ref{A}.
            Thus, $w\notin Y$.
		 If $v\in Y$, then $N^-(v)\subseteq \left(N^-[x_3]\cup\{x_2\}\right)$. Hence, by Claim \ref{A}, $x_2\in N^-(v)$, implying that $N^-(v)\subseteq\{x_2,x_3,w\}$, but if $N^-(v)=\{x_2,x_3\}$ or $N^-(v)=\{x_2,w\}$, then $D$ contains $J_2$ or $J_9$, respectively.
		Therefore, $v\notin Y$, and by symmetry we can conclude also that $u\notin Y$, a contradiction.
			
			 Assume now that $N^-(x_3)=\{v,w\}$. Observe that  $N^-(v)\cap N^-[X]\subseteq\{x_2,z\}$, then $v\notin Y$.
			If $u\in Y$, then $N^-(u)\subseteq\{x_2,x_3,z,w\}$, but it could not be neither $\{x_2,z\}$ nor $\{x_3,w\}$ (by Remark \ref{TT3}). If $N^-(u)=\{x_2,x_3\}$, then $D$ contains $J_3$; if $N^-(u)=\{x_2,w\}$, then $D$ contains $J_9$; if $N^-(u)=\{x_3,z\}$, then $D$ contains $J_7$; and if $N^-(u)=\{z,w\}$, then $D$ contains $J_{10}$. Therefore, $u\notin Y$.
			If $w\in Y$, then $N^-(w)\subseteq\{x_1,x_2,u,z\}$. Hence, by Remark  \ref{TT3} and Claim \ref{A}, $N^-(w)=\{u,z\}$ or $N^-(w)=\{u,x_2\}$, implying that $D$ contains $J_{12}$ or $J_5$, respectively. Therefore, $w\notin Y$.
			If $z\in Y$, then $N^-(z)\subseteq\left(N^-[x_3]\cup N^-(x_1)\right)$. Hence, by Claim \ref{A} and Remark \ref{TT3}, $N^-(z)=\{u,w\}$ or $N^-(z)=\{u,x_3\}$, yielding that $D$ contains $J_{11}$ or $J_6$, respectively. Hence, $z\notin Y$, a contradiction.

     \noindent Subcase 2.3: $|N^-[X]|=8$. In this case, $N^-(x_3)=\{t,w\}$ for $t,w\not \in N^-[x_1]\cup N^-[x_2]$.
  First, observe that if $Y\cap\{x_1,x_2\}=\emptyset$, then without loss of generality $t\in Y$, $x_1\in N^-(t)$, yielding that $N^-(t)=N^-(x_2)$, a contradiction  to   Claim \ref{A}.  Therefore $Y\cap\{x_1,x_2\}\neq\emptyset$. Hence, since $N^-[x_3]\cap(N^-[x_1]\cup N^-[x_2])=\emptyset$ it follows that $N^-[x_3]\subseteq N^-[y]\cup N^-[y']$, with $y,y'\in Y$, yielding by Claim \ref{B} that $x_3\in Y$.
         If $x_2\notin Y$, then $Y=\{x_1,x_3,y\}$ and  $\{x_2,z\}=N^-(y)$, which is a contradiction to Remark \ref{TT3}. Therefore, $Y=\{x_2,x_3,y\}$, yielding that $N^-(x_1)\subseteq N^-[y]$, contradicting Claim \ref{B}.

     \noindent Case 3: Suppose  there are exactly two arcs between the elements of $X$. Then, $|N^-[X]|=6,7$. Let us distinguish the following cases.

    \noindent Subcase 3.1: First,  assume that $(x_1x_2x_3)$ is a path of $D$.
            Then, $N^-(x_2)\cap N^-(x_3)=N^-(x_2)\cap N^-(x_1)=\emptyset$ by Remark \ref{TT3}.  Hence, $ N^-(x_2)=\{z,x_1\}$.

 \noindent Subsubcase 3.1.1: $|N^-[X]|=6$. Without loss of generality, we may assume that $N^-(x_3)=\{x_2,u\}$.
             Observe that if $u\in Y$, then $N^-(u)=\{x_2, z\}$,  a contradiction to Remark \ref{TT3},  and  then $u\not\in Y$.
            If $v\in Y$, then $x_2\notin N^-(v)$  again  by Remark \ref{TT3}. Hence, if $v\in Y$ then $N^-(v)=\{x_3,z\}$, yielding that $D$ contains $J_4$.   Therefore, $z\in Y$ and $|Y\cap X|=2$. By Remark \ref{TT3} and Claim \ref{A}, $N^-(z)=\{x_3,v\}$, implying that $D$ contains $J_3$.

  \noindent Subsubcase 3.1.2: $|N^-[X]|=7$. Then $N^-(x_3)=\{x_2,w\}$  for some $w\not\in N^-[x_1]\cup N^-[x_2]$.
                                         If $w\in Y$, then $N^-(w)\subseteq (N^-[x_1]\cup\{z\})$ and, by Claim \ref{A} and Remark~\ref{TT3}
                                  $z\in N^-(w)$ and $N^-(w)\subseteq\{u,v,z\}$.  This implies that $D$ contains $J_6$. Therefore, $w\notin Y$.
                                If $z\in Y$,  then $N^-(z)\subseteq N^-(x_1)\cup\{x_3,w\}$. Hence, by Claim \ref{A} and Remark \ref{TT3}, without loss of generality, $N^-(z)=\{v,w\}$ or $N^-(z)=\{v,x_3\}$, implying that $D$ contains $J_8$ or $J_2$, respectively. Therefore, $z\notin Y$.
                                If $u\in Y$,  then $N^-(u)\subseteq N^-[x_3]\cup\{z\}$.  By Claim \ref{A} and Remark \ref{TT3}, $N^-(u)=\{z,x_3\}$ or $N^-(u)=\{z,w\}$, yielding that $D$ contains $J_4$ or $J_5$, respectively. Therefore, $u\notin Y$ and, by symmetry  $v\notin Y$, hence, $Y\setminus X=\emptyset$, a contradiction.

    \noindent Subcase 3.2: Second, let us assume that $N^-(x_2)=\{x_1,x_3\}$. If $|N^-[X]|=6$, then
 without loss of generality suppose that  $N^-(x_3)=\{v,z\}$. Observe that $v\not\in Y$, otherwise $N^-(v)=\{x_2\}$. If $z\in Y$, then $N^-(z)\subseteq\{x_1,x_2,u\}$ and by Remark \ref{TT3}, $N^-(z)= \{x_2,u\}$, yielding that $D$ contains $J_3$. Hence, $z\not\in Y$, and reasoning similarly, $u\not\in Y$, a contradiction. If $|N^-[X]|=7$, then $N^-(x_3)=\{z,w\}$ for some $w\not \in \{x_1,x_2,x_3,u,v,z\}$. If $u\in Y$, then $N^-(u)\subseteq N^-[x_3]\cup\{x_2\}$, and  by Claim \ref{A} and Remark \ref{TT3}, $x_2\in N^-(u)$ and $N^-(u)\subseteq\{x_2,z,w\}$, implying that $D$ contains $J_3$. Therefore, $u\notin Y$. Analogously, $v,z,w\notin Y$, yielding that $Y\setminus X=\emptyset$, a contradiction.

    \noindent Subcase 3.3: Third, without loss of generality, let us assume that $(x_1,x_2),(x_1,x_3)\in A$. If $|N^-[X]|=6$, then $N^-(x_2)=\{x_1,z\}=N^-(x_3)$, which contradicts Claim \ref{Dif in-neigh of X}. Hence, $|N^-[X]|=7$. Let $N^-(x_2)=\{x_1,z\}$ and $N^-(x_3)=\{x_1,w\}$.  Observe that   we also may assume that there are exactly two arcs between the elements of $Y$ and  there is some $y\in Y$ satisfying the same as $x_1$, that is,  $N^+[y]\cap Y=Y-y$. Therefore, if $x_1\in Y$ we can assume that $Y=\{x_1,u,w\}$ and $N^+(u)\cap Y=\{x_1,w\}$. Then, $N^-(u)\subseteq\{x_3,x_2,z\}$ and by Remark \ref{TT3}, $x_3\in N^-(u)$, implying that $N^-(u)=\{x_3,x_2\}$ or $N^-(u)=\{x_3,z\}$ yielding that $D$ contains $J_2$ or $J_9$, respectively. Moreover, since $N^+(x_1)\cap N^-[X]=\{x_2,x_3\}$, it follows that
$Y\cap\{x_2,x_3\}\neq\emptyset$. Furthermore, by Claim \ref{comun}, $|Y\cap\{x_2,x_3\}|=1$. Without loss of generality, suppose  $Y\cap X=\{x_2\}$.
If $Y=\{x_2,z,u\}$, then $u\in N^+(z)$, yielding that $N^-(z)\subseteq\{v,x_3,w\}$. By Remark \ref{TT3}, $N^-(z)=\{v,w\}$ or $N^-(z)=\{v,x_3\}$, implying that $D$ contains $J_5$ or $J_4$, respectively.  If $Y=\{x_2,z,w\}$, then $z\in N^-(w)$ and $x_3\in N^-(z)$. Then, without loss of generality $u\in N^-(z)$ yielding that $D$ contains $J_3$.
Therefore, $z\notin Y$. If $Y=\{x_2, u,v\}$, then without loss of generality $x_3\in N^-(u)$ and $w\in N^-(v)$, implying that $D$ contains $J_3$. Finally, if $Y=\{x_2, u,w\}$, then $x_3\in N^-(u)$ and $v\in N^-(w)$, yielding that $D$ contains $J_2$.

\noindent Case 4: Suppose there are three arcs between the elements of $X$. Hence, $|N^-[X]|=6$ and since $D$ is $TT_3$-free, we may assume that $(x_1x_2x_3x_1)$ is a directed triangle. Then, $N^-(x_i)\cap N^-(x_j)=\emptyset$, for all $i\ne j$.
Let $N^-(x_1)=\{x_2,u\}$, $N^-(x_2)=\{x_3,v\}$ and $N^-(x_3)=\{x_1,z\}$.
Notice that if $z\in N^-(u)$ or $v\in N^-(u)$, then $D$ contains $J_2$ or $J_3$, respectively. Therefore, since $D$ is a $TT_3$-free oriented graph, $N^-(u)\cap N^-[X]\subseteq\{x_2\}$ and  $u\notin Y$. Observe that, by symmetry, we can conclude the same for $v$ and $z$, obtaining again a contradiction.
\end{proof}

\bibliographystyle{plain}
\setlength{\itemsep}{-2mm}
%\footnotesize

\end{document}